\documentclass[]{article}
\usepackage{amssymb,calc,enumerate,booktabs}
\usepackage{multirow}
\usepackage{etoolbox,rotating,graphicx}
\usepackage{savesym}
\savesymbol{AND}
\usepackage{algorithmic}
\usepackage{algorithm}
\usepackage{caption}
\usepackage{fixltx2e,url}
\usepackage{savesym}
\usepackage{wrapfig}
\usepackage{amsmath}
\usepackage[subrefformat=parens,labelformat=parens]{subcaption}
\usepackage{comment}
\savesymbol{theoremstyle}
\usepackage{amsthm}
\usepackage{arydshln}
\usepackage{color}
\usepackage{cite}
\usepackage{geometry}
\allowdisplaybreaks

\DeclareGraphicsExtensions{.pdf,.jpeg,.png}
\restoresymbol{theorem}{theoremstyle}
\newtheorem{theorem}{Theorem}

\newtheorem{lemma}[theorem]{Lemma}
\newtheorem{corollary}[theorem]{Corollary}

\newcommand{\Aa}{\mathcal{A}}
\newcommand{\xx}{\mathcal{X}}
\newcommand{\yy}{\mathcal{Y}}
\newcommand{\Gg}{\mathcal{G}}
\newcommand{\ff}{\mathcal{F}}
\newcommand{\R}{\mathbb{R}}

\newcommand{\mb}[1]{\mathbf{#1}}

\begin{document}

\title{The Bilinear Assignment Problem: Complexity and polynomially solvable special cases}

\author{
\sc{Ante \'Custi\'c}\thanks{{\tt acustic@sfu.ca}.
Department of Mathematics, Simon Fraser University Surrey, 250-13450 102nd AV, Surrey, British Columbia, V3T 0A3, Canada}
\and
\sc{Vladyslav Sokol}\thanks{{\tt vsokol@sfu.ca}.
School of Computer Science, Simon Fraser University, 8888 University Drive, Burnaby, British Columbia, V5A 1S6, Canada}
\and
\sc{Abraham P. Punnen}\thanks{{\tt apunnen@sfu.ca}. Department of Mathematics, Simon Fraser University Surrey,  250-13450 102nd AV, Surrey, British Columbia, V3T 0A3, Canada}
\and
\sc{Binay Bhattacharya}\thanks{{\tt binay@sfu.ca}.
School of Computer Science, Simon Fraser University, 8888 University Drive, Burnaby, British Columbia, V5A 1S6, Canada}}

\maketitle

\begin{abstract}
In this paper we study the {\it bilinear assignment problem} (BAP) with size parameters $m$ and $n$,  $m\leq n$.  BAP is a generalization of the well known quadratic assignment problem and the three dimensional assignment problem and hence NP-hard. We show that BAP cannot be approximated within a constant factor unless P=NP even if the associated quadratic cost matrix $Q$ is diagonal. Further, we show that BAP remains  NP-hard if $m = O(\sqrt[r]{n})$, for some fixed $r$, but is solvable in polynomial time if $m = O(\sqrt{\log n})$. When the rank of $Q$ is fixed, BAP is observed to admit FPTAS and when this rank is one, it is solvable in polynomial time under some additional restrictions. We then provide a necessary and sufficient condition for BAP to be equivalent to two linear assignment problems. A closed form expression to compute the average of the objective function values of all solutions is presented, whereas the median of the solution values cannot be identified in polynomial time, unless P=NP. We then provide polynomial time heuristic algorithms that find a solution with objective function value no worse than that of $(m-1)!(n-1)!$ solutions. However, computing a solution whose objective function value is no worse than that of $m!n!-\lceil\frac{m}{\beta}\rceil !\lceil\frac{n}{\beta}\rceil !$ solutions is NP-hard for any fixed rational number $\beta>1$.
\medskip

\end{abstract}

\section{Introduction}\label{sec:intro}

 Let  $Q=(q_{ijk\ell})$ be an  $m\times m\times n\times n$ array,  $C=(c_{ij})$ be an $m\times m$ matrix, and  $D=(d_{k\ell})$ be an $n\times n$ matrix. Then the {\it bilinear assignment problem} (BAP) is to
\begin{align}
	\text{Minimize} \qquad &\sum_{i=1}^m\sum_{j=1}^m\sum_{k=1}^n\sum_{\ell=1}^n q_{ijk\ell}x_{ij}y_{k\ell} + \sum_{i=1}^m\sum_{j=1}^m c_{ij}x_{ij} + \sum_{k=1}^n\sum_{\ell=1}^n d_{k\ell}y_{k\ell}  \nonumber\\
	\text{subject to}\quad \  \ & \sum_{j=1}^m x_{ij}=1 \qquad \qquad i=1,2,\ldots,m, \label{x1}\\
	&\sum_{i=1}^m x_{ij}=1 \qquad \qquad j=1,2,\ldots,m, \label{x2}\\
	&\sum_{\ell=1}^n y_{k\ell}=1 \qquad \qquad k=1,2,\ldots,n, \label{y1}\\
	&\sum_{k=1}^n y_{k\ell}=1 \qquad \qquad \ell=1,2,\ldots,n, \label{y2}\\
	&x_{ij},\ y_{k\ell}\in \{0,1\} \qquad  i,j=1,\ldots,m,\ \ k,\ell=1,\ldots,n. \label{int}
\end{align}

Let $\xx$ be the set of all $m\times m$  0-1 matrices satisfying \eqref{x1} and \eqref{x2} and $\yy$ be the set of all $n \times n$ 0-1 matrices satisfying \eqref{y1} and \eqref{y2}. Also, let $\ff$ be the set of all feasible solutions of BAP. Note that $|\ff|=m!n!$. An instance of the BAP is fully defined by the 3-tuple of cost arrays $(Q,C,D)$. Let $M=\{1,2,\ldots,m\}$ and $N=\{1,2,\ldots,n\}$. Without loss of generality we assume that $m\leq n$. The objective function of BAP is denoted by $f(\mathbf{x},\mathbf{y})$ where $\textbf{x}=(x_{ij})\in\xx$ and $\mathbf{y}=(y_{ij})\in\yy$. The quadratic part of the objective function, i.e.\@ $\sum_{i=1}^m\sum_{j=1}^m\sum_{k=1}^n\sum_{\ell=1}^n q_{ijk\ell}x_{ij}y_{k\ell}$, is denoted by $\bar{f}(\mb{x},\mb{y})$. It may be noted that in BAP, constraints (\ref{int}) can be replaced by $0\leq x_{ij}\leq 1$ and $0\leq y_{k\ell}\leq 1$ for $i,j=1,\ldots,m,$ and $ k,\ell=1,\ldots,n.$ This justifies the name bilinear assignment problem \cite{A68, K76}.
\medskip

BAP is closely related to the well known {\it quadratic assignment problem} (QAP) \cite{C98}. Let $Q^{\prime}=(q^{\prime}_{ijk\ell})$  be an $n\times n\times n\times n$ array. Then the QAP is defined as
\begin{align}
	\text{Minimize} \qquad &\sum_{i=1}^n\sum_{j=1}^n\sum_{k=1}^n\sum_{\ell=1}^n q^{\prime}_{ijk\ell}x_{ij}x_{k\ell} \nonumber\\
	\text{subject to}\quad \  \ & \sum_{j=1}^n x_{ij}=1 \qquad \qquad i=1,2,\ldots,n, \nonumber\\
	&\sum_{i=1}^n x_{ij}=1 \qquad \qquad j=1,2,\ldots,n, \nonumber\\
	&x_{ij}\in \{0,1\} \qquad \qquad i,j=1,\ldots,n. \nonumber
\end{align}

Note that $Q^{\prime}$ could be viewed as an $n^2\times n^2$ matrix. Let $Q^{\prime\prime}=Q^{\prime}+\alpha I$ where $I$ is the $n^2\times n^2$  identity matrix. It is well known that the optimal solution set for the QAP with cost matrix $Q^{\prime}$ and $Q^{\prime\prime}$ are identical. By choosing $\alpha$ appropriately, we could make $Q^{\prime\prime}$ either positive semidefinite or negative semidefinite. Thus without loss of generality we could assume that $Q^{\prime}$ in QAP is positive semidefinite~\cite{Burkard-book,C98}. Under this assumption, if $(\mb{x}^*,\mb{y}^*)$ is an optimal solution to the BAP instance $(Q^{\prime},O^n,O^n)$, where $O^n$ is an $n\times n$ zero matrix, then both $\mb{x}^*$ and $\mb{y}^*$ are optimal solutions to the QAP~\cite{x1,K76}. Thus BAP is a generalization of QAP.

Furthermore, the \emph{axial three-dimensional assignment problem} (3AP)~\cite{S00} is also a special case of BAP. Consider the $n\times n\times n$ cost array $A=(a_{ijk})$.  Then the 3AP is to find $\mb{x},\mb{y}\in \yy$ such that
\[
	\sum_{i=1}^n\sum_{j=1}^n\sum_{k=1}^n a_{ijk}x_{ij}y_{jk}
\]
is minimized.
Now it is easy to see that every 3AP instance with $A=(a_{ijk})$ can be reduced to a BAP instance $(Q,O^n,O^n)$ where $q_{ijk\ell}$ is equal to $a_{ij\ell}$ if $j=k$, and 0 otherwise. The above reduction of 3AP to BAP is a special case of a more general problem considered by Frieze~\cite{fr1}.
\medskip

Being a generalization of the well studied problems QAP and 3AP, applications of these models translate into applications of BAP. Further, Zikan~\cite{z2} used a model equivalent to BAP to solve track initialization in the multiple-object tracking problem, Tsui and Chang \cite{TC90, TC92} used BAP to model a dock door assignment problem, and Torki, Yajima and Enkawa~\cite{TYE96} used BAP to obtain heuristic solutions to QAP with a small rank $Q$. Later in this paper we show that the well known \emph{disjoint matchings problem}~\cite{F83} can be formulated as a BAP with special structure.
\medskip

To the best of our knowledge, there are no other works that have been reported in the literature on BAP. Thus it is interesting and relevant to investigate fundamental structural properties of the problem, which is the primary focus of our paper. More specifically, we study complexity, approximability, and polynomially solvable special cases of BAP.
\medskip

The paper is organized as follows.
In Section~\ref{sec:comp} we investigate the complexity of BAP and prove that it is NP-hard if $m= O(\sqrt[r]{n})$, for some fixed $r$, but is polynomially solvable if $m= O(\sqrt{\log n})$. Note that QAP is polynomially solvable if $Q^{\prime}$ is diagonal, but we show that BAP is NP-hard even if $Q$ is diagonal and $n=m$. Moreover, such BAP instances do not admit a polynomial time $\alpha$-approximation algorithm for any fixed $\alpha>1$, unless P=NP. Section~\ref{sec:comp} also deals with some special classes of BAP that are solvable in polynomial time. In particular, we provide an algorithm to solve BAP when  $Q$, observed as a matrix, is of rank one, and either $C$ or $D$ is a sum matrix (i.e., either $c_{ij}$'s or $d_{ij}$'s are of the form $s_i + t_j$, for some vectors $S=(s_i)$ and $T=(t_i)$).
Furthermore, we show that a BAP instance is equivalent to two linear assignment problems, if and only if $Q$ is of the form $q_{ijk\ell}=e_{ijk}+f_{ij\ell}+g_{ik\ell}+h_{jk\ell}$, for a natural definition of equivalancy. In this case, the problem can be solved in $O(n^3)$ time. In Section~\ref{sec:approx}, we analyze various methods for approximating an optimal solution. We show that there is a FPTAS for BAP if the rank of $Q$ is fixed, and present a rounding procedure that preserves the objective value.
Computing a solution whose objective function value is no worse than that of $m!n!-\lceil\frac{m}{\beta}\rceil !\lceil\frac{n}{\beta}\rceil !$ solutions is observed to be NP-hard for any fixed rational  number $\beta>1$.
We also obtain a closed form expression for computing the average of the objective function values of all feasible solutions, and present methods for finding a solution with the objective value guaranteed to be no worse than the average value. Such solutions are shown to have objective function value no worse than $(m-1)!(n-1)!$  alternative solutions. Finally, we note that a median solution cannot be found in polynomial time, unless P=NP. Here the median denotes a solution which is in the middle of the list of all feasible solutions sorted by their objective function value. Concluding remarks are given in Section~\ref{conclusion}.


\section{Complexity and polynomially solvable cases}\label{sec:comp}

Since BAP is a generalization of QAP, it is clearly strongly NP-hard.  It is well known that for any $\alpha >1$ the existence of a polynomial time $\alpha$-approximation algorithm for QAP implies P=NP~\cite{sago}.   The reduction from QAP to BAP discussed in the introduction section need not preserve approximation ratio. Let us now discuss another reduction from QAP to BAP that partially preserves approximation ratio.

Suppose $m=n$ and impose the additional restriction in BAP that $x_{ij}=y_{ij}$ for all $i,j$. The resulting problem is equivalent to QAP. Constraints $x_{ij}=y_{ij}$ can be enforced by modifying the entries of $Q,$ $C$ and $D$ without explicitly stating the constraints or changing the objective function values of the solutions that satisfy the new constants. For every $i,j$ we can change $c_{ij}$ to $c_{ij}+L$, $d_{ij}$ to $d_{ij}+L$ and $q_{ijij}$ to $q_{ijij}-2L$, for some large $L$. This will increase the objective value by
$\sum_{i,j=1}^n L(x_{ij}-2x_{ij}y_{ij}+y_{ij})=\sum_{i,j=1}^nL(x_{ij}-y_{ij})^2$, which forces $x_{ij}=y_{ij}$ in an optimal solution. Since the reduction described above preserves the objective values of the solutions that satisfy $x_{ij}=y_{ij}$, BAP inherits the approximation hardness of  QAP. That is, for any $\alpha > 1$, BAP does not have a polynomial time $\alpha$-approximation algorithm, unless P=NP.

More interestingly we now show that this non-approximability result for BAP  extends to the case even if $Q$ is a diagonal matrix. Note that QAP is polynomially solvable when $Q^{\prime}$ is diagonal.

\begin{theorem}\label{tc2}
BAP is NP-hard even if $Q$ is a diagonal matrix. Further,  for any $\alpha >1$ the existence of a polynomial time $\alpha$-approximation algorithm for BAP when $Q$ is a diagonal matrix implies P=NP.
\end{theorem}
\begin{proof}
To prove the theorem, we reduce the DISJOINT MATCHINGS problem to BAP. Input of the problem is given by a complete bipartite graph $K_{n,n}=(V_1,V_2,E)$ and two sets of edges $E_1,E_2 \subseteq E$. The goal is to decide if there exist $M_1\subseteq E_1$ and $M_2\subseteq E_2$, such that $M_1$ and $M_2$ are perfect matchings and $M_1\cap M_2=\emptyset$. DISJOINT MATCHINGS is shown to be NP-complete by Frieze~\cite{F83}. A simpler proof can be found in \cite{F97}.

Now let us assume that for some $\alpha>1$ there exists a polynomial time $\alpha$-approximation algorithm for BAP when $Q$ is a diagonal matrix. Consider an instance of DISJOINT MATCHINGS defined on $K_{n,n}=(V_1,V_2,E)$, the edge set $E_1$ and the edge set $E_2$, where $V_1=\{v_1,v_2,\ldots,v_n\}$ and $V_2=\{u_1,u_2,\ldots,u_n\}$. Let $(Q,C,D)$ be an instance of BAP where the $n\times n\times n \times n$ array $Q=(q_{ijk\ell})$ is the identity matrix when observed as an $n^2\times n^2$ matrix (i.e., $q_{ijk\ell}=1$ if $i=k$, $j=\ell$, and $0$ otherwise). Further, let $C=(c_{ij})$ and $D=(d_{ij})$ be $n\times n$ matrices given by
\begin{equation}\label{ered}
c_{ij}=\begin{cases}
	\frac{1}{\alpha+1} & \text{if } (v_i,u_j)\in E_1 \text{ and } i=1,\\
	0 & \text{if } (v_i,u_j)\in E_1  \text{ and } i\neq 1,\\
	1 & \text{otherwise,}
\end{cases}
\quad \text{ and } \quad
d_{ij}=\begin{cases}
	0 & \text{if } (v_i,u_j)\in E_2,\\
	1 & \text{otherwise.}
\end{cases}
\end{equation}

Now we show that our DISJOINT MATCHINGS instance is a ``yes" instance, if and only if there exists a solution $(\mb{x}^*,\mb{y}^*)$ of BAP on the instance $(Q,C,D)$ with the objective value $\frac{1}{\alpha+1}$. Assume that $(\mb{x}^*,\mb{y}^*)\in \ff$ is such that $f(\mb{x}^*,\mb{y}^*)=\frac{1}{\alpha+1}$. Since all costs are non-negative, and $c_{1j}\geq\frac{1}{\alpha+1}$ for all $j=1,\ldots,n$, it follows that $\sum_{i=1}^n\sum_{j=1}^n c_{ij}x^*_{ij}=\frac{1}{\alpha+1}$, which implies that $\mb{x}^*$ corresponds to a perfect matching $S_1$ which is a subset of $E_1$. Similarly, it must be that  $\sum_{i=1}^n\sum_{j=1}^n d_{ij}y^*_{ij}=0$, hence $\mb{y}^*$ corresponds to a perfect matching $S_2$ which is a subset of $E_2$. Furthermore, $f(\mb{x}^*,\mb{y}^*)=\frac{1}{\alpha+1}$ implies that $\sum_{i=1}^n\sum_{j=1}^n\sum_{k=1}^n\sum_{\ell=1}^n q_{ijk\ell}x^*_{ij}y^*_{k\ell}=0$. Then since $Q$ is an identity matrix, it follows that $S_1$ and $S_2$ are disjoint, hence our DISJOINT MATCHINGS instance is a ``yes" instance. Conversely, if we assume that our DISJOINT MATCHINGS instance is a ``yes" instance, then by the same arguments but in opposite direction, it follows that an optimal solution of the BAP instance $(Q,C,D)$ has the objective value equal to $\frac{1}{\alpha+1}$.

For every BAP instance $(Q,C,D)$ obtained from a DISJOINT MATCHINGS instance using the reduction described above, if the objective value of an arbitrary feasible solution is not $\frac{1}{\alpha+1}$, it will be at least $1$. Since $\alpha\frac{1}{\alpha+1}<1$, our $\alpha$-approximation algorithm for the BAP with diagonal matrices $Q$ can be used for solving the NP-complete DISJOINT MATCHINGS problem, which implies that P=NP.
\end{proof}

If we replace costs $\frac{1}{\alpha+1}$ in \eqref{ered} with $0$, we get a reduction from DISJOINT MATCHINGS to BAP instances with 0-1 costs. Therefore, BAP is NP-hard even when restricted to instances with 0-1 costs where $Q$ is the identity matrix.

Let us now examine the impact of the ratio between $m$ and $n$ on the tractability of BAP.
\begin{theorem}
	If $m= O(\sqrt{\log n})$ then BAP can be solved in polynomial time. However, for any fixed $r$, if $m=O(\sqrt[r]{n})$, then BAP is NP-hard.
\end{theorem}
\begin{proof}
	For a given $\mb{x}^*\in\xx$, one can find a $\mb{y}\in\yy$  which minimizes $f(\mb{x}^*,\mb{y})$ by investing $O(n^3)$ time. This is achieved by solving a \emph{linear assignment problem} (LAP) \cite{Burkard-book}  with cost matrix $H=(h_{k\ell})$ where $h_{k\ell}=d_{k\ell}+\sum_{i=1}^m\sum_{j=1}^m q_{ijk\ell}x^*_{ij}$. Hence, if $|\xx|=m!$ is a polynomial in $n$, then BAP can be solved in  polynomial time by enumerating $\xx$ and solving all corresponding LAPs. Since $\xx$ is a subset of the set of all $m\times m$ 0-1 matrices, it follows that $|\xx|< 2^{m^2}.$ Hence, if $m= O(\sqrt{\log n})$, then we get that $|\xx|= O(n^c)$, for some constant $c$. This proves the first part of the theorem.

Now we prove the second part. Let $(Q,C,D)$ be an arbitrary instance of BAP with size parameters $m$ and $n$. Given a constant $r$, consider the $n\times n \times n^r\times n^r$ cost array $\hat{Q}$, the $n\times n$ matrix $\hat{C}$ and the $n^r\times n^r$ matrix $\hat{D}$ defined as follows:
\[
\hat{q}_{ijk\ell}=\begin{cases}
	q_{ijk\ell} & \text{if }i,j\leq m \text{ and } k,\ell\leq n,\\
	0 & \text{otherwise,}
\end{cases}
\qquad
\hat{c}_{ij}=\begin{cases}
	c_{ij} & \text{if }i,j\leq m,\\
	0 & \text{if } i,j>m,\\
	L & \text{otherwise,}
\end{cases}
\qquad
\hat{d}_{k\ell}=\begin{cases}
	d_{k\ell} & \text{if }k,\ell\leq n,\\
	0 & \text{if } k,\ell>n,\\
	L & \text{otherwise,}
\end{cases}
\]
where $L$ is a large number. It is easy to see that the BAP instances $(Q,C,D)$ and $(\hat{Q},\hat{C},\hat{D})$ are equivalent in the sense that from an optimal solution to one problem, an optimal solution to other can be recovered in polynomial time. $(\hat{Q},\hat{C},\hat{D})$ satisfies the condition of the second statement of the theorem and it can be constructed in polynomial time, therefore the result follows.
\end{proof}


Polynomially solvable special cases of the QAP and the 3AP have been investigated by many researchers~\cite{Burkard-book,S00}. Since QAP and 3AP are special cases of the BAP, all of their polynomially solvable special cases can be mapped onto polynomially solvable special cases for the BAP.  Let us now consider some new polynomially solvable cases.

Note that the rows of $Q$ can be labeled using the ordered pairs $(i,j)$, $i,j=1,2,\ldots ,m$. Then the row of $Q$ represented by the label $(i,j)$ can be viewed as an $n\times n$ matrix $P^{ij}=(p^{ij}_{k\ell})$ where $p^{ij}_{k\ell}=q_{ijk\ell}$ for $k,\ell = 1,2, \ldots ,n$. A linear assignment problem (LAP) with cost matrix $P^{ij}$ has {\it constant value property} (CVP) if there exists a constant $\alpha_{ij}$ such that $\sum_{k=1}^n\sum_{\ell=1}^np^{ij}_{k\ell}y_{k\ell}=\alpha_{ij}$ for all $\mb{y}\in \yy$. Characterizations of cost matrices with CVP for various combinatorial optimization problems have been studied by different authors~\cite{BE79,CK16}. 

In the case when LAP with cost matrix $P^{ij}$ has CVP for all $i,j=1,2,\ldots, m$, we define $W=w_{ij}$ to be an $m\times m$ matrix where $w_{ij}=\alpha_{ij}+c_{ij}$.

\begin{theorem}\label{tp1}
If $P^{ij}$ satisfies CVP for all $i,j=1,2,\ldots ,m$, $\mb{x}^*$ is an optimal solution to the LAP with cost matrix $W$, and $\mb{y}^*$ is an optimal solution to the LAP with cost matrix $D$, then $(\mb{x}^*,\mb{y}^*)$ is an optimal solution to the BAP $(Q,C,D)$.
\end{theorem}
\begin{proof}
\begin{align*}
\min_{\mb{x}\in \xx,\mb{y}\in \yy}f(\mb{x},\mb{y})&=\min_{\mb{x}\in \xx,\mb{y}\in \yy}\left\{\sum_{i=1}^m\sum_{j=1}^m\sum_{k=1}^n\sum_{\ell=1}^n q_{ijk\ell}x_{ij}y_{k\ell} + \sum_{i=1}^m\sum_{j=1}^m c_{ij}x_{ij} + \sum_{k=1}^n\sum_{\ell=1}^n d_{k\ell}y_{k\ell}\right\}\\
&=\min_{\mb{x}\in \xx,\mb{y}\in \yy}\left\{\sum_{i=1}^m\sum_{j=1}^m\alpha_{ij}x_{ij} + \sum_{i=1}^m\sum_{j=1}^m c_{ij}x_{ij} + \sum_{k=1}^n\sum_{\ell=1}^n d_{k\ell}y_{k\ell}\right \}\\
&=\min_{\mb{x}\in \xx}\left\{\sum_{i=1}^m\sum_{j=1}^m(\alpha_{ij} + c_{ij})x_{ij}\right\} + \min_{\mb{y}\in \yy}\left\{\sum_{k=1}^n\sum_{\ell=1}^n d_{k\ell}y_{k\ell}\right \}
\end{align*}
Thus,  BAP decomposes into two LAPs and the result follows.
\end{proof}

An analogous result can be derived using CVP for columns of $Q$. We omit the details.
\medskip

A linear assignment problem (LAP) with cost matrix $P^{ij}$ has \emph{two value property} (2VP) if there exist some constants $\alpha_{ij}$ and $ \beta_{ij}$ such that $\sum_{k=1}^n\sum_{\ell=1}^np^{ij}_{k\ell}y_{k\ell}$ is equal to either $\alpha_{ij}$ or $\beta_{ij}$, for all $\mb{y}\in \yy$. Tarasov~\cite{tar81} gave a characterization of cost matrices for LAP having 2VP. In view of Theorem~\ref{tp1}, it would be interesting to explore the complexity of BAP when $P^{ij}$ satisfies 2VP, for all $i,j$. As shown below, the polynomial solvability of BAP does not extend if CVP of matrices $P^{ij}$ is replaced by 2VP.

\begin{theorem}
BAP remains strongly NP-hard even if $P^{ij}$ satisfies 2VP for all $i,j=1,2,\ldots, m$.
\end{theorem}
\begin{proof}
If $Q$ is a diagonal matrix with no zero entries on the diagonal, then it is easy to verify that $P^{ij}$ satisfies 2VP for all $i,j=1,2,\ldots ,m$. The result now follows from Theorem~\ref{tc2}.
\end{proof}

\subsection{Characterization of linearizable instances}

In Theorem~\ref{tp1} we observed that there are special cases of BAP that can be solved by solving two independent LAPs. We generalize this by characterizing instances of BAP that are `equivalent' to two independent LAPs. Let us first introduce some definitions.

An $m\times m\times n\times n$ cost array $Q=(q_{ijk\ell})$ associated with a BAP is said to be \emph{linearizable} if there exist an $m\times m$ matrix $A=(a_{ij})$ and an $n\times n$ matrix $B=(b_{ij})$, such that
\[
	\sum_{i=1}^m\sum_{j=1}^m\sum_{k=1}^n\sum_{\ell=1}^n q_{ijk\ell}x_{ij}y_{k\ell} = \sum_{i=1}^m\sum_{j=1}^m a_{ij}x_{ij} + \sum_{k=1}^n\sum_{\ell =1}^n b_{k\ell}y_{k\ell}
\]
for every $\mb{x}\in \xx$ and $\mb{y}\in \yy$. The matrices $A$ and $B$ are called the linearization of $Q$. 

Note that the collection of all linearizable $m\times m\times n\times n$ cost arrays forms a subspace of $\R^{m^2\times n^2}$. If $Q$ is linearizable, and a linearization (i.e.\@ matrices $A$ and $B$) is given, then the BAP instance $(Q,C,D)$ is solvable in $O(m^3+n^3)$ time for every $C$ and $D$. I.e., such BAP instances reduce to two LAPs with respective cost matrices $A+C$ and $B+D$. Therefore we say that an instance $(Q,C,D)$ of BAP is linearizable if and only if $Q$ is linearizable. Linearizable instances of QAP have been studied by various authors~\cite{KP11, PK13, CDW}. Corresponding properties for the \emph{quadratic spanning tree problem} (QMST) was investigated in \cite{CP15}. Before attempting to characterize linearizable instances of BAP, let us first establish some preliminary results.

\begin{lemma}\label{lmLinSuff}
If the cost array $Q=(q_{ijk\ell})$ of a BAP satisfies
	\begin{equation}\label{eqSumDec}
		 q_{ijk\ell}=e_{ijk}+f_{ij\ell}+g_{ik\ell}+h_{jk\ell}\qquad i,j\in N,\ k,\ell\in M,
	\end{equation}
	for some $m\times m\times n$ arrays $E=(e_{ijk})$, $F=(f_{ijk})$ and $m\times n\times n$ arrays $G=(g_{ijk})$, $H=(h_{ijk})$, then $Q$ is linearizable.
\end{lemma}
\begin{proof}
	Let $Q$ be of the form \eqref{eqSumDec}, for some  $E=(e_{ijk})$, $F=(f_{ijk})$, $G=(g_{ijk})$ and $H=(h_{ijk})$. Then for every $\mb{x}\in \xx$ and $\mb{y}\in \yy$ we have that
\begin{align}
	\bar{f}(\mb{x},\mb{y}) &= \sum_{i=1}^m\sum_{j=1}^m\sum_{k=1}^n\sum_{\ell=1}^n q_{ijk\ell}x_{ij}y_{k\ell}
	= \sum_{i=1}^m\sum_{j=1}^m \sum_{k=1}^n\sum_{\ell=1}^n\left(e_{ijk}+ f_{ij\ell}+g_{ik\ell}+ h_{jk\ell}\right)x_{ij}y_{k\ell}\nonumber\\
	&= \sum_{i=1}^m\sum_{j=1}^m \sum_{k=1}^n\sum_{\ell=1}^n\left(e_{ijk}+ f_{ij\ell}\right)x_{ij}y_{k\ell}  + \sum_{i=1}^m\sum_{j=1}^m\sum_{k=1}^n\sum_{\ell =1}^n \left(g_{ik\ell}+ h_{jk\ell}\right)x_{ij}y_{k\ell}\nonumber\\
	&= \sum_{i=1}^m\sum_{j=1}^m \left(\sum_{k=1}^n\sum_{\ell=1}^n\left(e_{ijk}+ f_{ij\ell}\right)y_{k\ell}\right) x_{ij} + \sum_{k=1}^n\sum_{\ell =1}^n \left(\sum_{i=1}^m\sum_{j=1}^m\left(g_{ik\ell}+ h_{jk\ell}\right)x_{ij}\right)y_{k\ell}. \label{eqLmLn1}	
\end{align}
Note that LAP given on a sum matrix $T$ (i.e.\@ of the form $t_{ij}=r_i+s_j$) has CVP. Moreover, the constant objective function value for such matrix is precisely $\sum_i r_i+\sum_j s_j$. Hence, for every $\mb{y}\in \yy$ and every $\mb{x}\in \xx$
\[ \sum_{k=1}^n\sum_{\ell=1}^n\left(e_{ijk}+ f_{ij\ell}\right)y_{k\ell}=\sum_{k=1}^n e_{ijk}+\sum_{\ell=1}^n  f_{ij\ell}=a_{ij} \mbox{ (say)}, \]
and
\[ \sum_{i=1}^m\sum_{j=1}^m\left(g_{ik\ell}+ h_{jk\ell}\right)x_{ij}=\sum_{i=1}^m g_{ik\ell}+\sum_{j=1}^m  h_{jk\ell}=b_{kl}\mbox{ (say)}. \]
Then, from \eqref{eqLmLn1} we have,
\[
	\sum_{i=1}^m\sum_{j=1}^m\sum_{k=1}^n\sum_{\ell=1}^n q_{ijk\ell}x_{ij}y_{k\ell} = \sum_{i=1}^m\sum_{j=1}^m a_{ij}x_{ij} + \sum_{k=1}^n\sum_{\ell =1}^n b_{k\ell}y_{k\ell},
\]
which proves the lemma.
\end{proof}

Later we will show that the sufficient condition \eqref{eqSumDec} is also  necessary for $Q$ to be linearizable. Arrays satisfying \eqref{eqSumDec} are called \emph{sum decomposable arrays with parameters $4$ and $3$} in \cite{CK16}, where general sum decomposable arrays were investigated.

To establish that condition \eqref{eqSumDec} is also necessary for linearizability of the associated BAP, we use the following lemma.

\begin{lemma}\label{lmConNec}
	If the cost array  $Q=(q_{ijk\ell})$ associated with a BAP satisfies
	\begin{equation}\label{eqConst}
		 \bar{f}(\mb{x},\mb{y})=\sum_{i=1}^m\sum_{j=1}^m\sum_{k=1}^n\sum_{\ell=1}^n q_{ijk\ell}x_{ij}y_{k\ell}=K
	\end{equation}
	for all $\mb{x}\in \xx$, $\mb{y}\in \yy$ and some constant $K$, then $Q$ must be of the form \eqref{eqSumDec}.
\end{lemma}
\begin{proof}
	Let $i,j\in M$ and $k,\ell\in N$ be chosen arbitrary such that $i,j,k,\ell\geq 2$. Also, let $\mb{x}^1,\mb{x}^2\in\xx$ be such that they only differ on index pairs from $\{1,i\}\times\{1,j\}$. In particular, let $x^1_{11}=x^1_{ij}=x^2_{1j}=x^2_{i1}=1$.
Similarly, let $\mb{y}^1,\mb{y}^2\in\yy$ be such that they only differ on index pairs from $\{1,k\}\times\{1,\ell\}$ and in particular, let $y^1_{11}=y^1_{k\ell}=y^2_{1\ell}=y^2_{k1}=1$.

Now let us assume that \eqref{eqConst} is true for every $\mb{x}\in\xx$ and $\mb{y}\in\yy$. Then it follows that \begin{equation}\label{eqConst2}
	\bar{f}(\mb{x}^1,\mb{y}^1)+\bar{f}(\mb{x}^2,\mb{y}^2)=\bar{f}(\mb{x}^1,\mb{y}^2)+\bar{f}(\mb{x}^2,\mb{y}^1).
\end{equation}
If we expand objective value expressions in \eqref{eqConst2} and cancel out identical parts from the left and right sides, we get
\[
	q_{1111}+q_{11k\ell}+q_{ij11}+q_{ijk\ell}+q_{1j1\ell}+q_{1jk1}+q_{i11\ell}+q_{i1k1}=
	q_{111\ell}+q_{11k1}+q_{ij1\ell}+q_{ijk1}+q_{1j11}+q_{1jk\ell}+q_{i111}+q_{i1k\ell},
\]
from which it follows that
\begin{equation}\label{eqConstAlign}\begin{aligned}
	q_{ijk\ell}=&\hspace{13pt} q_{ijk1}+q_{ij1\ell}+q_{i1k\ell}+q_{1jk\ell}\\
			&-q_{ij11}-q_{i1k1}-q_{i11\ell}-q_{1jk1}-q_{1j1\ell}-q_{11k\ell}\\
			&+q_{i111}+q_{1j11}+q_{11k1}+q_{111\ell}\\
			&-q_{1111}
\end{aligned}\end{equation}
for every $i,j,k,\ell\geq 2$. However, note that \eqref{eqConstAlign} also holds true when some of $i,j,k,\ell$ are equal to 1. Namely, if we replace one of $i,j,k,\ell$ with 1, everything cancels out. Hence \eqref{eqConstAlign} holds for all $i,j\in M$, $k,\ell\in N$.

Define the $m\times m\times n$ arrays $A=(a_{ijk})$, $B=(b_{ijk})$ and the $m\times n\times n$ arrays $C=(c_{ijk})$, $D=(d_{ijk})$ as
\begin{align*}
	a_{ijk} &= q_{ijk1}-\frac{1}{2}q_{ij11}-\frac{1}{2}q_{i1k1}-\frac{1}{2}q_{1jk1}+\frac{1}{3}q_{i111}+\frac{1}{3}q_{1j11}+\frac{1}{3}q_{11k1}-\frac{1}{4}q_{1111},\\
	b_{ij\ell} &= q_{ij1\ell}-\frac{1}{2}q_{ij11}-\frac{1}{2}q_{i11\ell}-\frac{1}{2}q_{1j1\ell}+\frac{1}{3}q_{i111}+\frac{1}{3}q_{1j11}+\frac{1}{3}q_{111\ell}-\frac{1}{4}q_{1111},\\
	c_{ik\ell} &= q_{i1k\ell}-\frac{1}{2}q_{i1k1}-\frac{1}{2}q_{i11\ell}-\frac{1}{2}q_{11k\ell}+\frac{1}{3}q_{i111}+\frac{1}{3}q_{11k1}+\frac{1}{3}q_{111\ell}-\frac{1}{4}q_{1111},\\
	d_{jk\ell} &= q_{1jk\ell}-\frac{1}{2}q_{1jk1}-\frac{1}{2}q_{1j1\ell}-\frac{1}{2}q_{11k\ell}+\frac{1}{3}q_{1j11}+\frac{1}{3}q_{11k1}+\frac{1}{3}q_{111\ell}-\frac{1}{4}q_{1111}.
\end{align*}
Then, from \eqref{eqConstAlign} if follows that
\[
	q_{ijk\ell}=a_{ijk}+b_{ij\ell}+c_{ik\ell}+d_{jk\ell},
\]
which proves the lemma.
\end{proof}

We are now ready to prove the characterization of linearization instances of BAP.

\begin{theorem}
	The cost array $Q=(q_{ijk\ell})$ of a BAP is linearizable if and only if it is of the form
	\begin{equation}\label{eqLinThm}
		 q_{ijk\ell}=e_{ijk}+f_{ij\ell}+g_{ik\ell}+h_{jk\ell}\qquad i,j\in N,\ k,\ell\in M,
	\end{equation}
	for some $m\times m\times n$ arrays $E=(e_{ijk})$, $F=(f_{ijk})$ and $m\times n\times n$ arrays $G=(g_{ijk})$, $H=(h_{ijk})$.
\end{theorem}

\begin{proof}
	Lemma~\ref{lmLinSuff} tells us that \eqref{eqLinThm} is a sufficient condition for $Q$ to be linearizable. To show that this is also a necessary condition we start by following the steps of the proof of Lemma~\ref{lmLinSuff} in reverse.

Let $Q$ be a linearizable cost array. That is, there exist some $A=(a_{ij})$ and $B=(b_{ij})$ such that
	\begin{equation}\label{eqLin}
	\bar{f}(\mb{x},\mb{y)}=\sum_{i=1}^m\sum_{j=1}^m\sum_{k=1}^n\sum_{\ell=1}^n q_{ijk\ell}x_{ij}y_{k\ell} = \sum_{i=1}^m\sum_{j=1}^m a_{ij}x_{ij} + \sum_{k=1}^n\sum_{\ell =1}^n b_{k\ell}y_{k\ell}
\end{equation}
 for every $\mb{x}\in \xx$ and $\mb{y}\in \yy$. For every $i,j\in M$ we can choose $2n$ numbers $\hat{e}_{ijk}$, $\hat{f}_{ijk}$, $k=1,\ldots, n$, such that $\sum_{k=1}^n (\hat{e}_{ijk}+\hat{f}_{ijk})=a_{ij}$. Furthermore, for every $k,\ell\in N$ we can choose $2m$ numbers $\hat{g}_{ik\ell}$, $\hat{h}_{ik\ell}$, $i=1,\ldots, m$, such that $\sum_{i=1}^m (\hat{g}_{ik\ell}+\hat{h}_{ik\ell})=b_{k\ell}$. Now \eqref{eqLin} can be expressed as
\begin{align}
	\bar{f}(\mb{x},\mb{y}) &= \sum_{i=1}^m\sum_{j=1}^m \left(\sum_{k=1}^n \hat{e}_{ijk}+\sum_{\ell=1}^n \hat{f}_{ij\ell}\right) x_{ij} + \sum_{k=1}^n\sum_{\ell =1}^n \left(\sum_{i=1}^m \hat{g}_{ik\ell}+\sum_{j=1}^m \hat{h}_{jk\ell}\right)y_{k\ell}\nonumber\\
 &= \sum_{i=1}^m\sum_{j=1}^m \left(\sum_{k=1}^n\sum_{\ell=1}^n\left(\hat{e}_{ijk}+ \hat{f}_{ij\ell}\right)y_{k\ell}\right) x_{ij} + \sum_{k=1}^n\sum_{\ell =1}^n \left(\sum_{i=1}^m\sum_{j=1}^m\left(\hat{g}_{ik\ell}+ \hat{h}_{jk\ell}\right)x_{ij}\right)y_{k\ell}\nonumber\\
		&= \sum_{i=1}^m\sum_{j=1}^m \sum_{k=1}^n\sum_{\ell=1}^n\left(\hat{e}_{ijk}+ \hat{f}_{ij\ell}\right)x_{ij}y_{k\ell}  + \sum_{i=1}^m\sum_{j=1}^m\sum_{k=1}^n\sum_{\ell =1}^n \left(\hat{g}_{ik\ell}+ \hat{h}_{jk\ell}\right)x_{ij}y_{k\ell}\nonumber\\
	&= \sum_{i=1}^m\sum_{j=1}^m \sum_{k=1}^n\sum_{\ell=1}^n\left(\hat{e}_{ijk}+ \hat{f}_{ij\ell}+\hat{g}_{ik\ell}+ \hat{h}_{jk\ell}\right)x_{ij}y_{k\ell}.\label{eqChar3}
\end{align}
Expression \eqref{eqChar3} implies that
\[
	\sum_{i=1}^m\sum_{j=1}^m \sum_{k=1}^n\sum_{\ell=1}^n\left(q_{ijk\ell}-\left(\hat{e}_{ijk}+ \hat{f}_{ij\ell}+\hat{g}_{ik\ell}+ \hat{h}_{jk\ell}\right)\right)x_{ij}y_{k\ell}=0 \qquad \text{ for all } \mb{x}\in \xx,\ \mb{y}\in \yy.
\]
Now from Lemma~\ref{lmConNec}, it follows that the $m\times m\times n\times n$ array with entries $q_{ijk\ell}-(\hat{e}_{ijk}+ \hat{f}_{ij\ell}+\hat{g}_{ik\ell}+ \hat{h}_{jk\ell})$ is of the form \eqref{eqLinThm}, which implies that $Q$ itself is of the form \eqref{eqLinThm}. This concludes the proof.
\end{proof}

A natural question that arises is the following: Given a BAP cost array $Q$, can we determine whether $Q$ is linearizable, and if so, how can we find its linearization matrices $A$ and $B$? The answer is positive. Namely, consider the system of linear equations given by \eqref{eqLinThm} where $e_{ijk}$, $f_{ij\ell}$, $g_{ik\ell}$ and $h_{jk\ell}$ are the unknowns. This system has $m^2n^2$ equations and $2m^2n+2mn^2$ unknowns. Then $Q$ is linearizable if and only if our system has a solution. Furthermore, the linearization is given by
\[
	a_{ij}:=\sum_{k=1}^n (e_{ijk}+f_{ijk}) \qquad \text{and} \qquad b_{k\ell}:=\sum_{i=1}^m (g_{ik\ell}+h_{ik\ell}).
\]

\begin{corollary}
	Let $Q$ be a cost array of the BAP. Then deciding whether $Q$ is linearizable and finding the linearization matrices $A$ and $B$, can be done in polynomial time.
\end{corollary}

\subsection{Cost array of rank one}

Recall that the cost  array $Q$ of BAP can be viewed as an $m^2\times n^2$ cost matrix. The rank of $Q$, when $Q$ is viewed as a matrix, is at most $r$ if there exist some $m\times m$ matrices $A^{^{p}}=(a_{ij}^p)$ and $n\times n$ matrices  $B^{^p}=(b_{ij}^p)$, $p=1,\ldots,r,$ such that
\begin{equation}\label{fact}
	q_{ijk\ell}=\sum_{p=1}^r a_{ij}^pb_{k\ell}^p
\end{equation}
for all $i,j\in M$, $k,\ell\in N$. We say that \eqref{fact} is a \emph{factored form} of $Q$. Then solving a BAP instance $(Q,C,D)$, where matrix $Q$ is of fixed rank $r$, is the problem of minimizing
\begin{equation}\label{BAPfixed}
f(\mb{x},\mb{y})=\sum_{p=1}^r \left(\sum_{i,j=1}^m a_{ij}^px_{ij}\right) \left(\sum_{k,\ell=1}^n b_{k\ell}^py_{k\ell}\right) + \sum_{i,j=1}^m c_{ij}x_{ij} + \sum_{k,\ell=1}^n d_{k\ell}y_{k\ell},
\end{equation}
such that $\mb{x}\in\xx$, $\mb{y}\in\yy$.

\begin{theorem}
	If the $m^2\times n^2$ matrix $Q$ is of rank one and is given in factored form, and  either $C$ or $D$ is a sum matrix,  then the corresponding BAP instance $(Q,C,D)$ is solvable in $O(m^3+n^3)$ time.
\end{theorem}
\begin{proof}
	Let $Q$ be of rank one, i.e.\@ $q_{ijk\ell}=a_{ij}b_{k\ell}$ for all $i,j\in M$, $k,\ell\in N$. Furthermore, without loss of generality, assume that $D$ is a sum matrix. Then there exist $n$-vectors $S=(s_i)$ and $T=(t_i)$ such that  $d_{ij}=s_i+t_j$. Note that the value $\sum_{k,\ell=1}^n d_{k\ell}y_{k\ell}$ is the same for all $\mb{y}\in\yy$, hence solving the BAP instance $(Q,C,D)$ is equivalent to solving the BAP instance $(Q,C,O^n)$, that is, we need to minimize
\begin{equation}\label{BAPone}
f'(\mb{x},\mb{y}):=\left(\sum_{i,j=1}^m a_{ij}x_{ij}\right) \left(\sum_{k,\ell=1}^n b_{k\ell}y_{k\ell}\right) + \sum_{i,j=1}^m c_{ij}x_{ij},
\end{equation}
subject to $\mb{x}\in \xx$ and $\mb{y}\in\yy$.

Let $\mb{x}^*\in\xx$ and $\mb{y}^*\in\yy$ be a solution that minimizes \eqref{BAPone}. Note that if $\sum_{i,j=1}^m a_{ij}x^*_{ij}> 0$, then $\mb{y}^*$ is an assignment that minimizes $\sum_{k,\ell=1}^n b_{k\ell}y_{k\ell}$. Analogously, if $\sum_{i,j=1}^m a_{ij}x^*_{ij}< 0$, then $\mb{y}^*$ is an assignment that maximizes $\sum_{k,\ell=1}^n b_{k\ell}y_{k\ell}$. If $\sum_{i,j=1}^m a_{ij}x^*_{ij}= 0$, then $\mb{y}^*$ can be arbitrary as it does not contributes to the objective value $f'(\mb{x}^*,\mb{y}^*)$.
Hence, we only need to consider assignments $\mb{y}_{\min},\mb{y}_{\max}\in \yy$ that minimize and maximize $\sum_{k,\ell=1}^n b_{k\ell}y_{k\ell}$, respectively. They are found in $O(n^3)$ time. Once $\mb{y}$ is fixed to $\bar{\mb{y}}\in \yy$, minimizing $f'(\mb{x},\bar{\mb{y}})$ reduces to solving the linear assignment problem
\[
f'(\mb{x},\bar{\mb{y}})=\sum_{i,j=1}^m \left(\bar{B}\cdot a_{ij}+c_{ij}\right)x_{ij},
\]
where $\bar{B}:=\sum_{k,\ell=1}^n b_{k\ell}\bar{y}_{k\ell}$. This can be done in $O(m^3)$ time, and we do it for $\bar{\mb{y}}=\mb{y}_{\min}$ and $\bar{\mb{y}}=\mb{y}_{\max}$. Better of the two will give us an optimal solution.
\end{proof}




\section{Approximations}\label{sec:approx}

In the previous section we identified some classes of BAP for which an optimal solution can be found in polynomial time. In this section we develop and analyze various heuristic algorithms for BAP. In our analysis, we use approximation ratio~\cite{sago} and domination ratio~\cite{gp} to measure the quality of a heuristic solution.

\subsection{A rounding procedure}


  Let $(\bar{\mb{x}},\bar{\mb{y}})$ be a feasible solution to the bilinear program  (BALP) obtained by relaxing constraints \eqref{int} of BAP  to $0\leq x_{ij} \leq 1$  for all $i,j$ and $0 \leq y_{k\ell}\leq 1$  for all $k,\ell$. As indicated earlier, BALP is equivalent to BAP in the sense that there exists an optimal solution $(\mb{x},\mb{y})$ to BALP where $\mb{x}\in \xx$ and $\mb{y}\in \yy$. But fractional solutions could also be optimal for BALP (although there always exists an optimal 0-1 solution). We present a simple rounding procedure such that from any solution $(\bar{\mb{x}},\bar{\mb{y}})$ of BALP we can obtain a solution $(\mb{x}^*,\mb{y}^*)$ of BAP such that $f(\mb{x}^*,\mb{y}^*)\leq f(\bar{\mb{x}},\bar{\mb{y}})$. (Note that the notation $f(\mb{x},\mb{y})$ was introduced for 0-1 vectors but naturally extends to any vectors.) This algorithm is useful in developing our FPTAS.
\medskip
   
Given an instance $(Q,C,D)$ and a BALP solution  $(\bar{\mb{x}},\bar{\mb{y}})$, find an optimal solution $\mb{x}^*$ of the LAP defined by the $m\times m$ cost matrix  $H=(h_{ij})$ where
\begin{equation}\label{eq:hround}
h_{ij}=c_{ij}+\sum_{k=1}^n\sum_{\ell=1}^{n}q_{ijk\ell}\bar{y}_{k\ell}.
\end{equation}
Then choose $\mb{y}^*$ to be an optimal solution of the LAP defined by the $n\times n$ cost matrix  $G=(g_{k\ell})$ where
\begin{equation}\label{eq:ground}
g_{k\ell}= d_{k\ell}+\sum_{i=1}^m\sum_{j=1}^{m}q_{ijk\ell}x^*_{ij}.
\end{equation}
The above rounding procedure is called \emph{round-$x$ optimize-$y$} (RxOy), as $\mb{y}^*$ is an optimal 0-1 assignment matrix when $\mb{x}$ is fixed at $\mb{x}^*$, and $\mb{x}^*$ is obtained by ``rounding" $\bar{\mb{x}}$. Naturally, \emph{round-$y$ optimize-$x$} (RyOx) procedure can be defined by swapping the order of operations on $\bar{\mb{x}}$ and $\bar{\mb{y}}$ (we omit the details). Since LAP can be solved in cubic running time, we have that the complexity of RxOy and RyOx procedures is $O(m^2n^2+m^3+n^3)$.

Similar rounding procedures for the unconstrained bipartite boolean quadratic program and bipartite quadratic assignment problem were investigated in \cite{PSK} and \cite{CP15}, respectively.

\begin{theorem}\label{thm:round}
	If a feasible solution $(\mb{x}^*,\mb{y}^*)$ of BAP is obtained by RxOy or RyOx procedure from a feasible solution $(\bar{\mb{x}},\bar{\mb{y}})$ of BALP, then $f(\mb{x}^*,\mb{y}^*)\leq f(\bar{\mb{x}},\bar{\mb{y}})$.
\end{theorem}
\begin{proof}
		Let $(\bar{\mb{x}},\bar{\mb{y}})$ be a feasible solution of BALP. Then
	\begin{align*}
		f(\bar{\mb{x}},\bar{\mb{y}})&=\sum_{i\in M}\sum_{j\in M}\sum_{k\in N}\sum_{\ell\in N} q_{ijk\ell}\bar{x}_{ij}\bar{y}_{k\ell} + \sum_{i\in M}\sum_{j\in M} c_{ij}\bar{x}_{ij} + \sum_{k\in N}\sum_{\ell\in N} d_{k\ell}\bar{y}_{k\ell}\\
		&=\sum_{i\in M} \sum_{j\in M}\left( \sum_{k\in N}\sum_{\ell\in N} q_{ijk\ell}\bar{y}_{k\ell} + c_{ij}\right) \bar{x}_{ij} + \sum_{k\in N}\sum_{\ell\in N} d_{k\ell}\bar{y}_{k\ell}\\
		&=\sum_{i\in M} \sum_{j\in M}h_{ij} \bar{x}_{ij} + \sum_{k\in N}\sum_{\ell\in N} d_{k\ell}\bar{y}_{k\ell},
	\end{align*}
where $h_{ij}$'s are defined as in \eqref{eq:hround}.
Since the constraint matrix of LAP is totally unimodular, we have that $\sum_{i\in M} \sum_{j\in M}h_{ij} \bar{x}_{ij}\geq \sum_{i\in M} \sum_{j\in M}h_{ij} x^*_{ij}$. Hence
 	\begin{align*}
	f(\bar{\mb{x}},\bar{\mb{y}})&\geq \sum_{i\in M}\sum_{j\in M}\left( \sum_{k\in N}\sum_{\ell\in N} q_{ijk\ell}\bar{y}_{k\ell} + c_{ij}\right) x^*_{ij} + \sum_{k\in N}\sum_{\ell\in N} d_{k\ell}\bar{y}_{k\ell} \\
	&= \sum_{i\in M}\sum_{j\in M} c_{ij}x^*_{ij} + \sum_{k\in N}\sum_{\ell\in N}\left( \sum_{i\in M}\sum_{j\in M} q_{ijk\ell}x^*_{ij} + d_{k\ell}\right) \bar{y}_{k\ell}\\
	&= \sum_{i\in M}\sum_{j\in M} c_{ij}x^*_{ij} + \sum_{k\in N}\sum_{\ell\in N}g_{k\ell} \bar{y}_{k\ell}
	\end{align*}
where $g_{ij}$'s are defined as in \eqref{eq:ground}. Again, we have that $\sum_{k\in N} \sum_{\ell\in N}g_{k\ell} \bar{y}_{k\ell}\geq \sum_{k\in N} \sum_{\ell\in N}g_{k\ell} y^*_{k\ell}$, hence
\begin{align*}
	f(\bar{\mb{x}},\bar{\mb{y}})&\geq \sum_{i\in M}\sum_{j\in M} c_{ij}x^*_{ij} + \sum_{k\in N}\sum_{\ell\in N}\left( \sum_{i\in M}\sum_{j\in M} q_{ijk\ell}x^*_{ij} + d_{k\ell}\right) y^*_{k\ell}\\
	&=\sum_{i\in M}\sum_{j\in M}\sum_{k\in N}\sum_{\ell\in N} q_{ijk\ell}x^*_{ij}y^*_{k\ell} + \sum_{i\in M}\sum_{j\in M} c_{ij}x^*_{ij} + \sum_{k\in N}\sum_{\ell\in N} d_{k\ell}y^*_{k\ell}\\
	&=f(\mb{x}^*,\mb{y}^*).
	\end{align*}
The proof for RyOx works in the same way.
\end{proof}

\subsection{FPTAS for BAP with fixed rank of $Q$}\label{sub:fptas}

Recall that BAP does not admit a polynomial time $\alpha$-approximation algorithm, unless P=NP. However, we now observe that in the case when the cost matrix $Q$ is of fixed rank, there exists an FPTAS. This follows from the results of Mittal and Schulz in \cite{MS12}, where the authors present an FPTAS for a class of related optimization problems.

\begin{theorem}[Mittal, Schulz \cite{MS12}]\label{thm:MS}
Let $k$ be a fixed positive integer, and let $\mathcal{OP}^k$ be an optimization problem
\begin{align*}\label{f}
	\min\quad &g(\mb{z}) = h(E_1^T\mb{z}, E_2^T\mb{z}, \ldots, E_k^T\mb{z})\\
	\text{s.t.}\quad &\mb{z}\in \mathcal{P},
\end{align*}
where $\mathcal{P}\subseteq \R^n$ is a poytope, and function $h\colon \mathbb{R}^k_+\rightarrow \mathbb{R}$ and vectors $E_i\in\R^n$, $i=1,\ldots,k$, are such that the conditions
\begin{itemize}
	\item[(i)] $h(\alpha) \leq h(\beta) \hspace{26pt} \forall \alpha, \beta \in \mathbb{R}^k_+, \text{ s.t. } \alpha_i \leq \beta_i \text{ for all } i = 1, 2, \ldots, k,$
	\item[(ii)] $h(\lambda \alpha) \leq \lambda^c h(\alpha) \hspace{12pt} \forall \alpha \in \mathbb{R}^k_+, \lambda > 1 \text{ and some constant } c$,
	\item[(iii)] $E_i^T \mb{z} > 0 \hspace{38pt} \text{ for } i = 1, 2, \ldots, k \text{ and } \mb{z} \in \mathcal{P},$
\end{itemize}
are satisfied. Then there exists an FPTAS for $\mathcal{OP}^k$.
\end{theorem}

For the description of the FPTAS see \cite{MS12}.

\begin{corollary}
Let $r$ be a fixed integer. Then there exists an FPTAS for the BAP on a cost matrix $Q$ of rank $r$ \textup{(}see \eqref{BAPfixed}\textup{)}, if $\sum_{i,j=1}^m a_{ij}^px_{ij}>0$, $\sum_{k,\ell=1}^n b_{k\ell}^py_{k\ell}>0$, $p=1,2,\ldots,r$, and $\sum_{i,j=1}^m c_{ij}x_{ij}>0$, $\sum_{k,\ell=1}^n d_{k\ell}y_{e\ell}>0$ are satisfied for all $\mb{x}\in\xx$, $\mb{y}\in\yy$.
\end{corollary}
\begin{proof}
	The relaxations of the problems described in the statement of the corollary fall into the class of problems $\mathcal{OP}^k$ described in Theorem~\ref{thm:MS}. Namely, given such instance $(Q,C,D)$ of BALP, we can express it as an $\mathcal{OP}^k$ where $k:=2r+1$, $h(E_1^T \mb{z},E_2^T \mb{z},\ldots,E_{2r+1}^T \mb{z}):=\sum_{i=1}^r (E_{2i-1}^T \mb{z}) (E_{2i}^T \mb{z}) + E_{2r+1}^T \mb{z}$, polytop $\mathcal{P}\subset \R^{m^2+n^2}$ is the convex hull of $\{\mb{x} \otimes \mb{y}: \mb{x} \in \xx, \mb{y} \in \yy\}$, and
\[
E_i:=
\begin{cases}
	A^{(i+1)/2}\otimes O^n & \text{ if } i \text{ is odd,}\\
	O^m\otimes B^{i/2} & \text{ if } i \text{ is even,}\\
	C\otimes D & \text{ if } i=2r+1,
\end{cases}
\]
where $A^{^p}=(a^p_{ij})$, $B^{^p}=(b^p_{ij})$, $O^\ell$ denotes the $\ell\times \ell$ null-matrix, and $\otimes$ denotes an operation of concatenating $m\times m$ and $n\times n$ matrices into $(m^2+n^2)$-vectors. Furthermore, condition $(iii)$ is mandated in the corollary, which then implies $(i)$. Condition $(ii)$ can be checked easily. Hence, according to Theorem~\ref{thm:MS}, there exist an FPTAS that approximates our BALP with fixed rank matrix $Q$.  Theorem~\ref{thm:round} implies that by rounding fractional FPTAS solutions using RxOy (or RyOx), we get an FPTAS for the corresponding BAP with fixed rank $Q$.
\end{proof}

\subsection{Domination analysis}

The negative result of Theorem~\ref{tc2} precludes potential success in developing approximation algorithms for BAP with constant approximation ratio, unless additional strong assumptions are made. We now show that domination ratio~\cite{gp} can be used to establish some performance guarantee for BAP heuristics.

 \emph{Domination analysis} has been successfully pursued by many researchers to provide performance guarantee of heuristics  for various combinatorial optimization problems~\cite{a1,an1,CP152,gp,gr1,gg1r,gutin3,gutin2,gg2r,h1,k1,kh1,p3,p2,PSK,rb1,sn1,sd2,sn3,t1,v1,z1}. Domination analysis is also linked to exponential neighborhoods~\cite{x1} and very large-scale neighborhood search~\cite{a14,MP09}. Domination analysis results similar to what is presented here have been obtained for the \emph{bipartite quadratic assignment problems}~\cite{CP152} and the \emph{bipartite boolean quadratic programs}~\cite{PSK}.
\medskip

Given an instance $(Q,C,D)$ of a  BAP, let $\Aa(Q,C,D)$ be the average of the objective function values of all feasible solutions.

\begin{theorem}\label{thm:av}
 ${\displaystyle 	\Aa(Q,C,D)=\frac{1}{mn}\sum_{i=1}^m\sum_{j=1}^m\sum_{k=1}^n\sum_{\ell=1}^n q_{ijk\ell}+ \frac{1}{m}\sum_{i=1}^m\sum_{j=1}^mc_{ij}+ \frac{1}{n}\sum_{k=1}^n\sum_{\ell=1}^nd_{k\ell}}$.
\end{theorem}

The proof of Theorem~\ref{thm:av} follows from the observation that a cost $q_{ijk\ell}$ appears in the objective value sum of $(\mb{x},\mb{y})$ if and only if $x_{ij}=1$ and $y_{k\ell}=1$, and there is exactly $(m-1)!(n-1)!$ such solutions $(\mb{x},\mb{y})$ in $\ff$. We omit details.
\medskip

Consider the fractional solution $(\mb{x},\mb{y})$ where $x_{ij}=1/m$ for all $i,j\in M$, and let $y_{ij}=1/n$ for all $i,j\in N$. Then $(\mb{x},\mb{y})$ is a feasible solution to BALP. It is also easy to see that $f(\mb{x},\mb{y})=\Aa(Q,C,D)$. From Theorem~\ref{thm:round} it follows that a solution $(\mb{x}^*,\mb{y}^*)$ of BAP constructed from $(\mb{x},\mb{y})$ by applying  RxOy or RyOx, satisfies $f(\mb{x}^*,\mb{y}^*)\leq \Aa(Q,C,D)$. Therefore, a BAP feasible solution $(\mb{x}^*,\mb{y}^*)$ that satisfies $f(\mb{x}^*,\mb{y}^*)\leq \Aa(Q,C,D)$ can be obtained in $O(m^2n^2+m^3+n^3)$ time using RxOy or RyOx procedure.
\medskip

A feasible solution $(\mb{x},\mb{y})$ of BAP is said to be no better than average if $f(\mb{x},\mb{y}) \geq A(Q,C,D)$. We will provide a lower bound for the number of feasible solutions that are no better than the average. Establishing corresponding results for QAP  is an open problem~\cite{an1,CP152,gutin2,sn3}.
Given an instance $(Q,C,D)$ of BAP, define
\[\Gg(Q,C,D)=\{(\mb{x},\mb{y})\in\ff\ \colon f(\mb{x},\mb{y})\geq A(Q,C,D)\}.\]

\begin{theorem}\label{thm:dom}
 $|\Gg(Q,C,D)|\geq (m-1)!(n-1)!$.
\end{theorem}
\begin{proof}
	Consider an equivalence relation $\sim$ on $\ff$, where $(\mb{x},\mb{y})\sim(\mb{x}',\mb{y}')$ if and only if there exist $a\in\{0,1,\ldots,m-1\}$ and $b\in\{0,1,\ldots,n-1\}$ such that $x_{ij}=x'_{i(j+a \mod m)}$ for all $i,j$, and $y_{k\ell}=y'_{k(\ell+b \mod n)}$ for all $k,\ell$.  Note that $\sim$ partitions $\ff$ into $(m-1)!(n-1)!$ equivalence classes, each of size $mn$. Fix $S$ to be one such equivalence class. For every $i,j\in M$ and $k,\ell \in N$, there are exactly $n$ solutions in $S$ for which $x_{ij}=1$, and there are exactly $m$ solutions in $S$ for which $y_{k\ell}=1$. Furthermore, for every $i,j\in M$ and $k,\ell\in N$, there is exactly one solution in $S$ for which $x_{ij}y_{k\ell}=1$. Therefore
\[
	\sum_{(\mb{x},\mb{y})\in S}f(\mb{x},\mb{y})=\sum_{\substack{(i,j,k,\ell)\in \\ M\times M\times N \times N}} q_{ijk\ell}+n\sum_{\substack{(i,j)\in\\ M\times M}}c_{ij}+ m\sum_{\substack{(k,\ell)\in\\ N\times N}}d_{k\ell}=mn\Aa (Q,C,D).
\]
Because $|S|=mn$, there must be at least one $(\mb{x},\mb{y})\in S$ such that $f(\mb{x},\mb{y})\geq \Aa(Q,C,D)$. There is one such solution for each of the $(m-1)!(n-1)!$ equivalent classes and this concludes the proof.
\end{proof}

The bound presented in Theorem~\ref{thm:dom} is tight. To see this, let cost arrays $Q,C,D$ be such that all of their elements are 0, except $q_{i'j'k'\ell'}=1$ for some fixed $i',j',k',\ell'$. The tightness follows from the fact that exactly one element from every equivalence class defined by $\sim$ is no better than average.
\medskip

The proof of Theorem~\ref{thm:dom} also provides us a way to construct an $(\mb{x},\mb{y})\in\ff$ such that $f(\mb{x},\mb{y})\leq\Aa(Q,C,D)$. We showed that in every equivalence class defined by $\sim$ there is a feasible solution with the objective function value greater than or equal to  $\Aa(Q,C,D)$. By the same reasoning it follows that in every such class there is a feasible solution with objective function value less than or equal to $\Aa(Q,C,D)$. For example, given $a\in M$, $b\in N$ let $(\mb{x}^a,\mb{y}^b)\in\ff$ be defined as
\[
x_{ij}^a=
\begin{cases}
	1 & \text{if } j=i+a \mod m, \\
	0 & \text{otherwise}
\end{cases}\quad \text{and}\quad
y_{k\ell}^b=
\begin{cases}
	1 & \text{if } \ell=k+b \mod n, \\
	0 & \text{otherwise}.
\end{cases}
\]
Then $(\mb{x}^{a_1},\mb{y}^{b_1})\sim(\mb{x}^{a_2},\mb{y}^{b_2})$ for every $a_1,a_2\in M$ and $b_1,b_2\in N$, and
$$f(\mb{x}^a,\mb{y}^b)=\sum_{i\in M,k\in N} q_{i(i+a\hspace{-3pt} \mod m)k(k+b\hspace{-3pt} \mod n)}+\sum_{i\in M}c_{i(i+a\hspace{-3pt} \mod m)}+\sum_{k\in N}d_{k(k+b\hspace{-3pt} \mod n)}.$$

\begin{corollary}\label{cor:ab_av}
	For any instance $(Q,C,D)$ of BAP 
	\[\min_{a\in M,b\in N}\{f(\mb{x}^a,\mb{y}^b)\}\leq\Aa(Q,C,D)\leq\max_{a\in M,b\in N}\{f(\mb{x}^a,\mb{y}^b)\}.\]
\end{corollary}
Note that any equivalence class defined by $\sim$ can be used to obtain the type of inequalities above. Corollary~\ref{cor:ab_av} provides another way to find a feasible solution to BAP with objective function value no worse than $\Aa(Q,C,D)$ in $O(m^2n^2)$ time. This is a running time improvement when compared to the approach using RxOy or RyOx described above.
\medskip

Interestingly, unlike the average, computing the median of the objective function values of feasible solutions of BAP is NP-hard. The proof follows by a reduction from PARTITION PROBLEM. We omit details.

For a given instance $I$ of an optimization problem, let $\mb{x}^\Gamma\in\ff(I)$ be a solution produced by an algorithm $\Gamma$, where $\ff(I)$ denotes the set of feasible solutions for the instance $I$. Let $\Gg^\Gamma(I)=\{\mb{x}\in \ff(I)\ \colon f(\mb{x})\geq f(\mb{x}^\Gamma)\}$ where $f(\mb{x})$ denotes the objective function value of $\mb{x}$. Furthermore, let $\mathcal{I}$ be the collection of all instances of the problem with a fixed instance size. Then
\[
	\inf_{I\in\mathcal{I}}\left|\Gg^\Gamma(I)\right| \text{ \ and \ }  \inf_{I\in\mathcal{I}}\frac{|\Gg^\Gamma(I)|}{|\ff(I)|},
\]
are called \emph{domination number} and \emph{domination ratio} of $\Gamma$, respectively~\cite{a1,gp}.

As we discussed above, procedure RxOy (and RyOx) and Corollary~\ref{cor:ab_av} give us two BAP heuristics for which the domination number $(m-1)!(n-1)!$ and the domination ratio $\frac{1}{mn}$ is guarantied by Theorem~\ref{thm:dom}.
Lastly, we give an upper bound on the  domination ratio for any polynomial time heuristic algorithm for the BAP. The result can be shown following the main idea in \cite{PSK,CP15}, and hence we omit the proof.

\begin{theorem}
	For any fixed rational number $\beta>1$ no polynomial time algorithm for BAP can have domination number greater than $m!n!-{\lceil\frac{m}{\beta}\rceil!}{\lceil\frac{n}{\beta}\rceil!}$, unless P=NP.
\end{theorem}

\section{Conclusion}\label{conclusion}

We presented a systematic study of complexity aspects of BAP defined on the data set $(Q,C,D)$ and size parameters $m$ and $n$. BAP generalizes the well known quadratic assignment problem, the three dimensional assignment problem, and the disjoint matching problem. We show that BAP is  NP-hard if $m = O(\sqrt[r]{n})$, for some fixed $r$, but is solvable in polynomial time if $m = O(\sqrt{\log n})$. Further, we establish that BAP cannot be approximated within a constant factor unless P=NP even if $Q$  is diagonal; but when the rank of $Q$ is fixed, BAP is observed to admit FPTAS. When the  rank of $Q$  is one and $C$ or $D$ is a sum matrix, BAP is shown to be solvable in polynomial time. In contrast, QAP with a diagonal cost matrix is just the linear assignment problem which is solvable in polynomial time. We also provide a characterization of BAP instances equivalent to two linear assignment problems and this yields a rich class of polynomially solvable special cases of QAP.  Various results leading to performance guarantee of heuristics from the domination analysis point of view are presented. In particular, we showed that a feasible solution with objective function value no worse than that of $(m-1)!(n-1)!$ solutions can be identified efficiently, whereas  computing a solution whose objective function value is no worse than that of $m!n!-\lceil\frac{m}{\beta}\rceil !\lceil\frac{n}{\beta}\rceil !$ solutions is NP-hard for any fixed rational number $\beta>1$. As a by-product, we have a closed form expression to compute the average of the objective function values of all solutions, whereas the  median of the solution values cannot be identified in polynomial time, unless P=NP.

We are currently investigating efficient heuristic algorithms for BAP from an experimental analysis point of view and the results will be reported in a sequel. We hope that our work will inspire other researchers to explore further on the structural properties and algorithms for this interesting and versatile optimization model.


\section*{Acknowledgment}

This research work was supported by an NSERC discovery grant and an NSERC discovery accelerator supplement awarded to Abraham P. Punnen and an NSERC discovery grant awarded to Binay Bhattacharya.


\begin{thebibliography}{11}

\bibitem{a14} R.K. Ahuja, O. Ergun, J.B. Orlin, and A.P. Punnen, Very large scale neighborhood search: theory, algorithms and applications, \textit{Approximation Algorithms and Metaheuristics}, T. Gonzalez (ed), CRC Press, 2007.

\bibitem{a1} N. Alon, G. Gutin and M. Krivelevich, Algorithms with large domination ratio, {\it Journal on Algorithms} 50 (2004) 118--131.

\bibitem{A68} M. Altman, Bilinear programming, {\it Bullentin de l' Acad\'{e}mie Polonaise des Sciences} 16 (1968) 741--746.

\bibitem{an1} E. Angel and V. Zissimopoulos,  On the quality of local search for the quadratic assignment problem, {\it Discrete Applied Mathematics}  82 (1995) 15--25.

\bibitem{BE79} X. Berenguer, A characterization of linear admissible transformations for the $m$-traveling salesman problem, {\it European Journal of Operational Research} 3 (1979) 232--238.

\bibitem{Burkard-book}
R.E. Burkard, M. Dell'Amico, and S. Martello, Assignment Problems,
SIAM, Philadelphia (2009).

\bibitem{C98} E. \c{C}ela,  The Quadratic Assignment Problem: Theory and Algorithms, Kluwer Academic Publishers, Dordrecht (1998).

\bibitem{CDW} E. \c{C}ela, V.G. De\u\i neko and G.J. Woeginger,  Linearizable special cases of the QAP, {\it Journal of Combinatorial optimization} 31 (2016) 1269--1279.

\bibitem{CK16} A. \'Custi\'c  and B. Klinz, The constant objective value property for multidimensional assignment problems, \textit{Discrete Optimization} 19 (2016) 23--35.

\bibitem{CP15} A. \'Custi\'c  and A.P. Punnen,  Average value of solutions of the bipartite quadratic assignment problem and linkages to domination analysis, arXiv:1512.02709.

\bibitem{CP152} A. \'Custi\'c and A.P. Punnen, Characterization of the linearizable instances of the quadratic minimum spanning tree problem, arXiv:1510.02197.

\bibitem{x1} V.G. De\u\i neko and G.J. Woeginger, A study of exponential neighborhoods for the travelling salesman problem and for the quadratic assignment problem, {\it Mathematical programming} 87 (2000) 519--542.

\bibitem{F97} D.~Fon-Der-Flaass, Arrays of distinct representatives - a very simple NP-complete problem, {\it Discrete Mathematics} 171 (1997) 295--298.

\bibitem{F83} A.M. Frieze, Complexity of a {$3$}-dimensional assignment problem, {\it European Journal of Operational Research} 13 (1983) 161--164.

\bibitem{fr1} A.M. Frieze, A bilinear programming formulation of the 3-dimensional assignment problem, {\it Mathematical Programming} 7 (1974) 376--379.

\bibitem{gp}F. Glover and A.P. Punnen, The travelling salesman problem: new solvable cases and linkages with the development of approximation algorithms, {\it Journal of the Operational Research Society} 48 (1997) 502--510.

\bibitem{gr1} L.K. Grover, Local search and the local structure of NP-complete problems, {\it Operations Research Letters} 12 (1992) 235--243.

\bibitem{gg1r} G. Gutin, T. Jensen and A. Yeo, Domination analysis for minimum multiprocessor scheduling, \emph{Discrete Applied Mathematics} 154 (2006) 2613--2619.

\bibitem{gutin3} G. Gutin, A. Vainshtein and A. Yeo, Domination analysis of combinatorial optimization problems,  {\it Discrete Applied Mathematics} 129 (2003) 513--520.

\bibitem{gutin2} G. Gutin and A. Yeo, Polynomial approximation algorithms for the TSP and the QAP with a factorial domination number, {\it Discrete Applied Mathematics} 119 (2002) 107--116.

\bibitem{gg2r} G. Gutin and A. Yeo, TSP tour domination and Hamilton cycle decompositions of regular graphs, \emph{Operations Research Letters} 28 (2001) 107--111.

\bibitem{h1} R. Hassin and S. Khuller, z-Approximations, {\it Journal of Algorithms} 41 (2001) 429--442.

\bibitem{K76} H. Konno, Maximization of a convex quadratic function under linear constraints, {\it Mathematical programming} 11 (1976) 117--127.


\bibitem{KP11} S.N. Kabadi and A.P. Punnen, An $O(n^4)$ algorithm for the QAP linearization problem, {\it Mathematics of Operations Research} 36 (2011) 754--761.

\bibitem{k1} D. Kuhn and D. Osthus, Hamilton decompositions of regular expanders: a proof of {K}elly's conjecture for large tournaments, \emph{Advances in Mathematics} 237 (2013) 62--146.

\bibitem{kh1} A.E. Koller and S.D. Noble, Domination analysis of greedy heuristics for the frequency assignment problem, \emph{Discrete Mathematics} 275 (2004) 331--338.

\bibitem{MP09} A. Mitrovi\'c-Mini\'c, and A.P. Punnen, Local search intensified: Very large-scale variable neighborhood search for the multi-resource generalized assignment problem, {\it Discrete Optimization} 6 (2009) 370--377.

\bibitem{MS12} S. Mittal, and A.S. Schulz, An FPTAS for optimizing a class of low-rank functions over a polytope, \emph{Mathematical Programming} 141 (2012) 103--120.

\bibitem{p3} A.P. Punnen and S.N. Kabadi, Domination analysis of some heuristics for the asymmetric traveling salesman problem, {\it Discrete Applied Mathematics} 119 (2002) 117--128.

\bibitem{p2} A.P. Punnen, F. Margot and S.N. Kabadi, TSP heuristics: domination analysis and complexity, {\it Algorithmica} 35 (2003) 111--127.

\bibitem{PK13} A.P. Punnen and S.N. Kabadi, A linear time algorithm for the Koopmans-Beckman QAP linearization and related problems, {\it Discrete Optimization} 10 (2013) 200--209.

\bibitem{PSK} A.P. Punnen, P. Sripratak and D. Karapetyan, Average value of solutions for the bipartite boolean quadratic programs and rounding algorithms, {\it Theoretical Computer Science} 565 (2015) 77--89.


\bibitem{rb1}V.I. Rublineckii, Estimates of the accuracy of procedures in the traveling salesman problem, {\it Numerical Mathematics and Computer Technology} 4 (1973) 18--23 (in Russian).

\bibitem{sago} S. Sahni and T.F. Gonzalez, P-complete approximation problems, {\it Journal of the ACM} 23 (1976) 555--565.

\bibitem{sn1} V. Sarvanov and N. Doroshko, The approximate solution of the traveling salesman problem by a local algorithm that searches neighborhoods of exponential cardinality in quadratic time, {\it Software: Algorithms and Programs} 31 (1981) 8--11 (in Russian).

\bibitem{sd2} V. Sarvanov and N. Doroshko, The approximate solution of the traveling salesman problem by a local algorithm that searches neighborhoods of factorial cardinality in cubic time, {\it Software: Algorithms and Programs} 31 (1981) 11--13 (in Russian).

\bibitem{sn3} V. Sarvanov, The mean value of the functional in sampling problems, {\it Vestsi Akademii Navuk BSSR. Seryya Fizika-Matematychnykh Navuk} 139 (1978)  51--54.

\bibitem{S00} F.C.R. Spieksma, Multi index assignment problems: complexity, approximation, applications, \emph{Nonlinear assignment problems}, 1--12, Kluwer Academic Publishers, Dordrecht (2000).

\bibitem{tar81}  S.P. Tarasov, Properties of the trajectories of the appointments problem and the travelling-salesman problem,  \textit{USSR Computational Mathematics and Mathematical Physics} 21 (1981) 167--174.

\bibitem{TYE96}  A. Torki, Y. Yajima and T. Enkawa, A low-rank bilinear programming approach for sub-optimal solution of the quadratic assignment problem,  \textit{European Journal of Operational Research} 94 (1996) 384--391.

\bibitem{TC90} L.Y. Tsui, C.-H. Chang, A microcomputer based decision support tool for assigning dock doors in freight yards, {\it Computers \& Industrial Engineering} 19 (1990) 309--312.

\bibitem{TC92} L.Y. Tsui, C.-H. Chang, An optimal solution to a dock door assignment problem, {\it Computers \& Industrial Engineering} 23 (1992) 283--286.

\bibitem{t1} Y. Twitto, Dominance guarantees for above-average solutions, {\it Discrete Optimization} 5 (2008) 563--568.

\bibitem{v1} V.G. Vizing, Values of the target functional in a priority problem that are majorized by the mean value, {\it Kibernetika} 5 (1973) 76--78.

\bibitem{z1} E. Zemel, Measuring the quality of approximate solutions to zero-one programming problems, {\it Mathematics of Operations Research} 6 (1981) 319--332.

\bibitem{z2} K. Zikan, Track Initialization in the Multiple-Object Tracking Problem, Technical report SOL-88-18, Systems optimization laboratory, Stanford University, 1988.

\end{thebibliography}
\end{document}